\documentclass{amsart}

\usepackage[foot]{amsaddr}
\usepackage{hyperref}
\usepackage{amsfonts,amsmath,amssymb,amsthm,amscd,mathtools,mathrsfs,mathabx}
\usepackage{graphicx}
\usepackage{tikz}

\newtheorem{theorem}{Theorem}[section]
\newtheorem{lemma}[theorem]{Lemma}
\newtheorem{proposition}[theorem]{Proposition}
\newtheorem{corollary}[theorem]{Corollary} 
\newtheorem{definition}[theorem]{Definition}

\newtheorem{conjecture}[theorem]{Conjecture}

\newcommand{\swap}[2]{\xleftrightarrow{#1 \quad #2}}
\newcommand\Bcal{{\mathcal B}}
\newcommand\Hcal{{\mathcal H}}
\newcommand\Rcal{{\mathcal R}}
\DeclareMathOperator{\cl}{cl}
\DeclareMathOperator{\type}{type}

\begin{document}

\title{White's Conjecture for Paving Matroids}
\author{Yu-Chuan Yu, Chi Ho Yuen}
\address{Department of Applied Math., National Yang Ming Chiao Tung University, Taiwan}
\email{\url{yuyuch0303@gmail.com}, \url{chyuen@math.nctu.edu.tw}}
\subjclass[2020]{Primary: 05B35; Secondary: 13F65}

\maketitle
\begin{abstract}
White's conjecture asserts that any two tuples of matroid bases that have the same multi-set union can be transformed from one to another by symmetric exchanges; it also implies that the toric ideals of matroids are generated by the binomials encoding these exchanges. We prove White's conjecture for the class of paving matroids. Our strategy is to generalize the inductive argument using circuit-hyperplane relaxations in the recent work of Han et al. to stressed hyperplane relaxations.
\end{abstract}

\section{Introduction}

Matroids are combinatorial structures that capture the notion of independence in various mathematical areas from graph theory to algebraic geometry. A standard characterization of matroids is the {\em symmetric exchange property} of bases: given two distinct bases $A,B$ of a matroid $M$ and $a\in A\setminus B$, there exists $b\in B\setminus A$ such that both $(A\cup\{b\})\setminus\{a\}, (B\cup\{a\})\setminus\{b\}$ are bases. A natural question, first due to White \cite{White1980}, is a multi-basis, multi-step generalization of the said property. 

\begin{conjecture} \label{conj:White}
    Given two ordered tuples of bases $(B_1,\ldots, B_n)$ and $(B'_1,\ldots,B'_n)$ with equal multi-set union, there exists an {\em exchange sequence} $(B_1^{(0)},\ldots,B_n^{(0)})=(B_1,\ldots,B_n), (B_1^{(1)},\ldots,B_n^{(1)}),\ldots,(B_1^{(\ell)},\ldots,B_n^{(\ell)})=(B'_1,\ldots,B'_n)$ of tuples of bases such that, within each step indexed by $0<t\leq\ell$, some pair $B_i^{(t)}, B_j^{(t)}$ is obtained from $B_i^{(t-1)}, B_j^{(t-1)}$ by a symmetric exchange (and other bases remain the same).
\end{conjecture}

Conjecture~\ref{conj:White} is one of the most fundamental and well-known problems in matroid theory. Despite its elementary description, it is of interest to combinatorial commutative algebra and algebraic geometry: the conjecture describes the generators of the {\em toric ideals of matroids}, which include the ideals defining the {\em torus orbit closures} in Grassmannians \cite[Section~13.2]{MichalekSturmfels}.

\begin{conjecture} \label{conj:IM}
    Let $M$ be a matroid on ground set $E$ with set of bases $\Bcal$. Consider the kernel $I_M$ of the map $f_M:\mathbb{K}[y_B:B\in\Bcal]\rightarrow \mathbb{K}[x_e:e\in E]$ given by $y_B\mapsto \prod_{e\in B}x_e$. Then $I_M$ is generated by binomials of the form $y_Ay_B-y_{A'}y_{B'}$, where $A', B'$ are obtained from $A,B$ by a symmetric exchange.
\end{conjecture}

This formulation leads to several algebraic variants of White's conjecture \cite{Haws2009,Herzog2002}, and establishes further connections with the polyhedral geometry of {\em matroid polytopes} \cite{backman2023}. For example, the weak White's conjecture asserts that $I_M$ is generated by quadrics, but we refer to Conjecture~\ref{conj:White} as White's conjecture here, and say ``White's conjecture in degree $n$'' for Conjecture~\ref{conj:White} with a fixed $n$.

White's conjecture and almost all of its variants remain open in general\footnote{Two notable variants that are proven for all matroids are (1) Conjecture~\ref{conj:IM} up to saturation \cite{Lason2014} and (2) the existence of regular unimodular triangulations of matroid polytopes \cite{backman2023}.}, but several special cases have been proven. We refer the reader to the recent paper by Han, Micha\l ek, and Weigert \cite{han2025white} (which we discuss in more details in the next subsection) for a list of related works, and only highlight the works by Bonin and B\'{e}rczi--Schwarcz:

\begin{theorem} \cite{BONIN2013} \label{thm:Bonin}
    White's conjecture is true for {\em sparse paving matroids}.
\end{theorem}

\begin{theorem} \cite{Berczi2024} \label{thm:Berczi_Schwarcz}
    White's conjecture in degree 2 is true for {\em split matroids}. Moreover, $\ell$ can be chosen to be $\leq\min\{r,r-|B_1\cap B'_1|+1\}$.
\end{theorem}

The main contribution of our paper is to settle White's conjecture for the class of {\em paving matroids}, which are matroids whose circuits are all of size $r$ or $r+1$.

\begin{theorem} \label{thm:A}
    White's conjecture is true for paving matroids.
\end{theorem}

Every sparse paving matroid is paving whereas every paving matroid is split, so our result extends Theorem~\ref{thm:Bonin} and makes a step towards solving White's conjecture for split matroids (and more) completely.

The class of paving matroids, together with sparse paving matroids and split matroids, are interesting in the context of White's conjecture for several reasons. On one hand, it has been conjectured that asymptotically almost all matroids are paving (respectively, sparse paving or split), see \cite[Conjecture~1.6]{MNWW2011} and \cite[Conjecture~15.5.8]{Oxley}. Therefore, assuming this conjecture, our result confirms White's conjecture for a very large class of matroids (that properly contains all sparse paving matroids and Steiner systems). On the other hand, paving matroids can be characterized using their matroid polytopes \cite{Joswig2017}, which connects to the algebraic/geometric aspect of White's conjecture, and more generally, the modern trends in matroid theory.

\subsection{Overview of the Strategy}

Our strategy is inspired by the aforementioned work by Han et al. They proved that the validity of White's conjecture is preserved by {\em circuit-hyperplane relaxation} and its inverse operation.

\begin{theorem} \cite{han2025white}
    Let $M$ be a matroid with a circuit-hyperplane $H$ and let $\widetilde{M}$ be the relaxation of $M$ at $H$. Then White's conjecture is true for $M$ if and only if it is true for $\widetilde{M}$.
\end{theorem}

Since every sparse paving matroid is related to a uniform matroid by a sequence of circuit-hyperplane relaxations, and White's conjecture is trivially true for the latter, they gave an alternative proof of Theorem~\ref{thm:Bonin}.

The operation of circuit-hyperplane relaxation can be generalized to {\em stressed hyperplane relaxation} (relaxation for short), which was first considered by Bonin and de Mier \cite{bonin2008} before further developed by Ferroni, Nasr, and Vecchi \cite{ferroni2023stressed}. Every paving matroid is related to a uniform matroid by a sequence of relaxations, and our argument is by induction on the length of such a relaxation sequence. Part of our results remain true for arbitrary relaxations (regardless of $M$ is paving or not), particularly for higher degree. In the next theorem, we consider the condition 
\begin{center}
    (*): {\em Any two tuples (of bases of $M$) of degree 2\! or\! 3 that are connected by a 2-step exchange sequence in $\widetilde{M}$ can be connected by an exchange sequence in $M$.}
\end{center}

\begin{theorem} \label{thm:B}
    Let $M$ be a matroid with a stressed hyperplane $H$ and let $\widetilde{M}$ be the relaxation of $M$ at $H$. Then the followings are true:
    \begin{enumerate}
        \item if White's conjecture holds for $M$, then White's conjecture holds for $\widetilde{M}$;
        \item if White's conjecture holds for $\widetilde{M}$ and {\rm (*)} holds for $M$, then White's conjecture holds for $M$.
    \end{enumerate}
\end{theorem}

This theorem extends \cite[Theorem~2.13]{han2025white}, but the proof is more complicated as there are now many bases of $\widetilde{M}$ not in $M$, so the original induction (on the number of occurrences of these bases in an exchange sequence) and degree reduction procedure there are not enough.
Instead, we apply a more careful induction and modification procedure to the exchange sequence so to create the 2-step sequences that (*) can be applied on; we put no effort on optimizing $\ell$ as in Theorem~\ref{thm:Berczi_Schwarcz}. 

We prove the validity of (*) for paving matroids\footnote{As mentioned above, the degree 2 case is true more generally for split matroids, but we include our own proof here as a simpler illustration of our argument.} in Proposition~\ref{prop:XY_XY} and~\ref{prop:XYZ_XYZ}. Same as the above, their proofs are more difficult than their counterpart in \cite{han2025white}, as we need substantial work to make sure a basis of $\widetilde{M}$ we consider is also a basis of $M$ (e.g., in a circuit-hyperplane relaxation, performing an arbitrary basis exchange in $\widetilde{M}$ on the unique new basis produces a basis of $M$, but in our case more careful choice is needed). We sketch our high level strategy here.

\begin{enumerate}
    \item We may assume $E$ is the union of the bases, and by the 2-step exchange sequence in the relaxation, at most $4$ elements are being exchanged;
    \item we may assume that various candidate exchange sequences only exchanging these elements are not valid in the original matroid $M$, which implies that certain subsets considered in these sequences are circuits and they span distinct (stressed) hyperplanes $H$'s of $M$;
    \item we may assume certain pairs of these hyperplanes each covers $E$, or else we may use an uncovered element to produce a valid exchange sequence in $M$;
    \item the sum of $|H\cap H'|$'s over all pairs is now simultaneously small because of the property of stressed hyperplanes (Proposition~\ref{prop:SH_intersection}) and large because of the covering assumption, which yields a contradiction.
\end{enumerate}

\subsection*{Acknowledgement}
The second author thanks Mateusz Micha\l ek for suggesting this problem, as well as the helpful discussion and encouragement; he also thanks Spencer Backman for inspiring his interest on White's conjecture, and Joseph Bonin for pointing out \cite{bonin2008}. Both authors are partially supported by the Ministry of Science and Technology of Taiwan project MOST 114-2115-M-A49-015-MY3.

\section{Background}

We assume basic familiarity of matroid theory from the reader, a standard reference on the subject is \cite{Oxley}.
Unless otherwise specified, $M$ denotes a matroid of rank $r$ on ground set $E$. The set of bases of $M$ is denoted by $\Bcal$, and the closure operation is denoted by $\cl$. The fundamental circuit of $B\in\Bcal$ with respect to an element $e\not\in B$ is denoted by $C(B,e)$; we always take fundamental circuit in $M$.

\begin{definition}
    A matroid $M$ is {\em paving} if every circuit of $M$ is of size $r$ or $r+1$. Equivalently, any subset of size $\leq r-1$ is independent.
\end{definition}

The following is straightforward from the definition; in fact, it is not difficult to see that $M$ is paving if and only if it is $U_{2,2}\oplus U_{0,1}$-free.

\begin{proposition} \label{prop:minor-closed}
    The class of paving matroids is minor-closed.
\end{proposition}

\begin{definition} \label{def:SHP}
    A hyperplane $H$ of $M$ is {\em stressed} if every subset of size $r$ is circuit.
\end{definition}

\begin{theorem} \cite{bonin2008, ferroni2023stressed} \label{thm:SHR}
    Let $H$ be a stressed hyperplane of $M$. Then 
    $$\Bcal\cup\{S: S\subset H, |S|=r\}$$
    is the collection of bases of a matroid $\widetilde{M}$, known as the {\em relaxation} of $M$ at $H$.
\end{theorem}

A circuit-hyperplane relaxation is thus a special case of this operation where the stressed hyperplane $H$ is just a circuit. We only consider relaxation at a stressed hyperplane of size $\geq r$, so $M\neq\widetilde{M}$ (in particular, $M$ is not uniform).

The importance of stressed hyperplanes in the study of paving matroids is the following crucial theorem.

\begin{theorem} \cite{ferroni2023stressed} \label{thm:PM_SHR}
    Every hyperplane of a paving matroid $M$ is stressed. Moreover, the relaxation of $M$ at any of its stressed hyperplanes is again paving. Hence, every paving matroid can be related to the uniform matroid $U_{r,E}$ by a sequence of relaxations.
\end{theorem}

We collect a few simple properties of stressed hyperplanes and paving matroids.

\begin{proposition} \label{prop:SH_intersection}
    Let $H,H'$ be two distinct hyperplanes in a paving matroid $M$. Then $|H\cap H'|\leq r-2$.
\end{proposition}

\begin{proof}
    If $|H\cap H'|\geq r-1$, then $H\cap H'$ is of rank at least (indeed, exactly) $r-1$, which uniquely determines its closure as a hyperplane.
\end{proof}

\begin{proposition} \label{prop:type1}
    Let $H$ be a stressed hyperplane, $S\subset H$ be a subset of size $r-1$ and $x\not\in H$. Then $S\cup \{x\}$ is a basis of $M$.
\end{proposition}

\begin{proof}
    $S$ is an independent set of rank $r-1$, and $x$ is an element not in its closure $H$, so their union is a rank $r$ independent set, that is, a basis.
\end{proof}

\subsection{Convention of Notation}

We often write $X+ x_1\ldots x_k-y_1\ldots y_l$ or $X-y_1\ldots y_l+x_1\ldots x_k$ for $(X\cup\{x_1,\ldots,x_k\})\setminus\{y_1,\ldots,y_l\}$. Here $k,l$ are always small enough that we can list out the elements explicitly; moreover, we guarantee $y_1,\ldots,y_l$ are elements of $X$ whereas $x_1,\ldots,x_k$ are not, with the exception that $x_i=y_j$ (which we indicate the possibility explicitly). We use $\uplus$ for multi-set union.
A symmetric exchange from $(A,B)$ to $(A':=A-a+b,B':=B-b+a)$ is written as
$$
\begin{array}{c c c}
   A   & \swap{a}{b} & A'\\
   B   & \swap{b}{a} & B'
\end{array}.
$$

We denote by $\Bcal(H,x):=\{S\cup\{x\}:S\subset H, |S|=r-1\}$ the collection of bases in Proposition~\ref{prop:type1}. We often verify a subset $B$ in an exchange sequence is a basis by simply stating ``$B\in\Bcal(H,x)$'', if we believe it is clear to the reader that $B,H,x$ satisfy the assumption in the proposition.

When we work with a matroid with a single distinguished stressed hyperplane $H$, we often further omit $H$ and write $B\in\Bcal(x)$. In which case, we say a size $r$ subset $S$ is of type $|S\setminus H|$. By definition and Proposition~\ref{prop:type1}, a type 0 subset is a basis of the relaxation $\widetilde{M}$ but not $M$, a type 1 subset is necessarily a basis of both, and we only consider type $\geq 2$ subsets that are bases of $M$ (hence $\widetilde{M}$).

\subsection{Two Lemmas on Types}

We state a lemma on the exchange property of type 0 and type 1 subsets. As a bonus, the mild modification where we take $H=E$ in the proof justifies the claim that White's conjecture is trivial for uniform matroids.

\begin{lemma}\label{lem:type1-WC}
    Let $M$ be a matroid with a stressed hyperplane $H$.  
    Let $(B_1,\dots,B_n)$ and $(B'_1,\dots,B'_n)$ be two tuples of type $\leq 1$ subsets with the same multi-set union.
    Then there exists a sequence of symmetric exchanges transforming $(B_1,\dots,B_n)$ into $(B'_1,\dots,B'_n)$ such that all intermediate subsets are of type $\leq 1$.
\end{lemma}

\begin{proof}
First, we align all elements not in $H$ by applying symmetric exchanges. Suppose by induction, $B_i \setminus H = B'_i \setminus H$ for $i=1,\ldots,m-1$ thus 
\begin{equation}\label{eq:notHequal}
    \biguplus_{i\geq m} (B_i \setminus H)=\biguplus_{i\geq m} (B'_i \setminus H),
\end{equation}
but $B_m \setminus H \neq B'_m \setminus H$, say $a\in B_m \setminus H$ (reverse the role of $B_m,B'_m$ if needed). If $B'_m$ is of type 0 (respectively, contains $b\not\in H$), then there must exist some $k>m$ such that $B_k$ is also of type 0 (respectively, contains $b$) by (\ref{eq:notHequal}), and we can exchange $a\in B_m$ with an arbitrary element from $B_k\setminus B_m$ (respectively, $b\in B_k$), which is non-empty as $a\in B_m\setminus B_k$. Such an operation preserves the property that every subset is of type $\leq 1$.

Next, suppose by induction, $B_i \cap H = B'_i \cap H$ for $i=1,\ldots,m-1$ thus 
\begin{equation}\label{eq:Hequal}
    \biguplus_{i\geq m} (B_i \cap H)=\biguplus_{i\geq m} (B'_i \cap H).
\end{equation}
We then induct on $|B_m\triangle B'_m|$. If we are in the base case $B_m=B'_m$, then we can move on to the next pair. So suppose there exists $t\in (B'_m\setminus B_m) \cap H$. By (\ref{eq:Hequal}), there exists $k>m$ such that $t\in B_k$, and if $((B_m\setminus B'_m)\setminus B_k)\cap H\neq\emptyset$, we can pick $s$ from it and exchange $s,t$ between $B_m,B_k$, which reduces $|B_m\triangle B'_m|$ without modifying subsets whose indices are $<m$, nor changing the type of any subset.

Otherwise, we still pick some $s\in (B_m\setminus B'_m) \cap H$. The fact that $s\not\in B'_m$ and (\ref{eq:Hequal}) imply there exists $l>m, l\neq k$ such that $s\not\in B_l$, and we may assume $t\not\in B_l$ or else replace $k$ by $l$ in the above. Since $|B_k\cap H|, |B_l\cap H|$ differ by at most 1, while $s,t\in B_k, s,t\not\in B_l$, there exists $p\in (B_l\setminus B_k)\cap H$. Now we first exchange $t,p$ between $B_k, B_l$ before exchanging $s,t$ between $B_m$ and the updated $B_l$. This again reduces $|B_m\triangle B'_m|$ without changing other hypotheses.
\end{proof}

\begin{corollary}\label{cor:type1-WC-bases}
    Under the hypotheses of Lemma~\ref{lem:type1-WC}, if all $B_i$'s and $B'_i$'s are type~1 \emph{bases} of $M$, then there exists a sequence of symmetric exchanges transforming $(B_1,\dots,B_n)$ into $(B'_1,\dots,B'_n)$ such that all intermediate subsets are type~1 bases.
\end{corollary}

\begin{proof}
    Apply Lemma~\ref{lem:type1-WC} to obtain an exchange sequence whose intermediate subsets are of type $\leq 1$, they must be all of type 1 by considering total type.
\end{proof}

We also use the following easy lemma multiple times in this paper.

\begin{lemma} \label{lem:type02}
    Let $M$ be a matroid with a stressed hyperplane $H$.
    Let $X$ be a type 0 subset and $Y$ a basis of $M$ of type $\geq 2$. Let $a\in Y\setminus H\subset Y\setminus X$. Then there exists $s\in X\setminus Y\subset H$  such that $X-s+a, Y-a+s$ are bases of $M$.
\end{lemma}

\begin{proof}
    Since both $X,Y$ are bases of $\widetilde{M}$, we can pick $s\in X\setminus Y$ such that $X-s+a, Y-a+s$ are bases of $\widetilde{M}$. But the former is of type 1 and the latter is of type $\geq 1$, so they are bases of $M$.
\end{proof}

\section{Degree 2} \label{sec:deg2}

We show the degree 2 part of (*) for any paving matroid $M$, and the degree 2 part of Theorem~\ref{thm:B}(2) in this section. Together we prove White's conjecture in degree 2 for paving matroids.

\begin{proposition} \label{prop:XY_XY}
    Let $M$ be a paving matroid and fix a hyperplane $H$ of size $\geq r$. Let $\widetilde{M}$ be the relaxation of $M$ at $H$. Let $X,Y,X',Y'$ be 4 bases of $M$ (hence $\widetilde{M}$) such that there exists an exchange sequence from $(X,Y)$ to $(X',Y')$ in $\widetilde{M}$ as
\begin{equation} \label{eq:deg2}
\begin{array}{c c c c c}
   X   & \swap{a}{s} & X'' & \swap{t}{b} & X'\\
   Y   & \swap{s}{a} & Y'' & \swap{b}{t} & Y'
\end{array}.
\end{equation}

Then there exists an exchange sequence from $(X,Y)$ to $(X',Y')$ in $M$.
\end{proposition} 

\begin{proof}
    We are done if $X'',Y''$ are bases of $M$. On the other hand, they cannot be both of type 0, or else $X,Y\subset X''\cup Y''$ are also of type 0. Therefore we may assume $X''$ is of type 0 and $Y''$ is not. We may further assume $|\{a,b,s,t\}|=4$, otherwise, either $(X,Y)=(X',Y')$, or they only differ by a single symmetric exchange.
    
    By induction on $|E|$ (none of the argument in this paper considers matroids with ground sets larger than $E$), we may assume White's conjecture is true for any smaller paving matroids, thus without loss of generality $E=X\sqcup Y$: we may delete all elements outside of $X\cup Y$ and contract all elements in $X\cap Y$ if it is not, the minor is still paving by Proposition~\ref{prop:minor-closed}, while standard facts on matroid bases with respect to minor operations imply that we can lift an exchange sequence in the minor (guaranteed by White's conjecture) back to the original matroid.

    We write $X=\widetilde{X}+at, Y=\widetilde{Y}+bs$ (hence $X'=\widetilde{X}+bs,Y'=\widetilde{Y}+at$), and call $H=\cl(X'')=\cl(\tilde{X}+st)$ as $H_{X,st}$; note that $a,b\not\in H_{X,st}$, or else $X,Y$ are of type 0 and cannot be bases of $M$.\\
    
    {\bf Step 1}: We may assume $t\not\in C(Y'=\widetilde{Y}+at,s)$, otherwise $\widetilde{Y}+as\in\Bcal$ and
    $$
    \begin{array}{c c c c c}
        \widetilde{X}+at   & \swap{a}{b} & \widetilde{X}+bt & \swap{t}{s} & \widetilde{X}+bs\\
        \widetilde{Y}+bs   & \swap{b}{a} & \widetilde{Y}+as & \swap{s}{t} & \widetilde{Y}+at
    \end{array}
    $$
    is a valid exchange sequence in $M$. Here $\widetilde{X}+bt\in\Bcal(H_{X,st}, b)$.

    Therefore, we may assume $C(\widetilde{Y}+at,s)\subset\widetilde{Y}+as$, which must be an equality as the right hand side is already of size $r$. Consider the hyperplane $H_{Y,as}:=\cl(\widetilde{Y}+as)$ this circuit spans, we have $b,t\not\in H_{Y,as}$, or else $Y,Y'$ cannot be bases of $M$.
    
    Similarly, we may assume $s\not\in C(Y=\widetilde{Y}+bs,t)$, otherwise $\widetilde{Y}+bt\in\Bcal$ and
    $$
    \begin{array}{c c c c c}
        \widetilde{X}+at   & \swap{t}{s} & \widetilde{X}+as & \swap{a}{b} & \widetilde{X}+bs\\
        \widetilde{Y}+bs   & \swap{s}{t} & \widetilde{Y}+bt & \swap{b}{a} & \widetilde{Y}+at
    \end{array}
    $$
    is a valid exchange sequence in $M$ with $\widetilde{X}+as\in\Bcal(H_{X,st}, a)$. So $C(\widetilde{Y}+bs,t)=\widetilde{Y}+bt$ spans a hyperplane $H_{Y,bt}$ that does not contain $a,s$.

    Since $H_{Y,as}\cap H_{Y,bt}\supset\widetilde{Y}$, which is already of size $r-2$, by Proposition~\ref{prop:SH_intersection}, it must be an equality.\\

    {\bf Step 2}: We may assume $\widetilde{X}+ab$ is a circuit, or otherwise
    $$
    \begin{array}{c c c c c}
        \widetilde{X}+at   & \swap{t}{b} & \widetilde{X}+ab & \swap{a}{s} & \widetilde{X}+bs\\
        \widetilde{Y}+bs   & \swap{b}{t} & \widetilde{Y}+st & \swap{s}{a} & \widetilde{Y}+at
    \end{array}
    $$
    is a valid exchange sequence in $M$ with $\widetilde{Y}+st\in\Bcal(H_{Y,as}, t)$. Let $H_{X,ab}$ be the hyperplane it spans. We have $s,t\not\in H_{X,ab}$ and $H_{X,st}\cap H_{X,ab}=\widetilde{X}$ similarly.\\

    {\bf Step 3}: We may assume $E=H_{X,st}\cup H_{X,ab}$, or otherwise we may pick some $p\in E\setminus (H_{X,st}\cup H_{X,ab})$, which is necessarily in $\widetilde{Y}$, and
    $$
    \begin{array}{c c c c c c c}
        \widetilde{X}+at   & \swap{a}{p} & \widetilde{X}+tp & \swap{t}{b} & \widetilde{X}+bp & \swap{p}{s} & \widetilde{X}+bs\\
        \widetilde{Y}+bs   & \swap{p}{a} & \widetilde{Y}-p+abs & \swap{b}{t} & \widetilde{Y}-p+ast & \swap{s}{p} & \widetilde{Y}+at
    \end{array}
    $$
    is a valid exchange sequence in $M$. Here $\widetilde{X}+tp\in\Bcal(H_{X,st},p), \widetilde{X}+bp\in\Bcal(H_{X,ab},p)$, and $\widetilde{Y}-p+abs=(\widetilde{Y}-p+as)+b\in\Bcal(H_{Y,as},b), \widetilde{Y}-p+ast=(\widetilde{Y}-p+as)+t\in\Bcal(H_{Y,as},t)$.

    Similarly, we may assume $E=H_{Y,as}\cup H_{Y,bt}$, or otherwise we may pick some $q\in E\setminus (H_{Y,as}\cup H_{Y,bt})\subset\widetilde{X}$, and
    $$
    \begin{array}{c c c c c c c}
        \widetilde{X}+at   & \swap{q}{s} & \widetilde{X}-q+ast & \swap{t}{b} & \widetilde{X}-q+abs & \swap{a}{q} & \widetilde{X}+bs\\
        \widetilde{Y}+bs   & \swap{s}{q} & \widetilde{Y}+bq & \swap{b}{t} & \widetilde{Y}+tq & \swap{q}{a} & \widetilde{Y}+at
    \end{array}
    $$
    is a valid exchange sequence in $M$. Here $\widetilde{Y}+bq, \widetilde{Y}+tq\in\Bcal(H_{Y,bt},q)$, and $\widetilde{X}-q+ast=(\widetilde{X}-q+st)+a\in\Bcal(H_{X,st},a), \widetilde{X}-q+abs=(\widetilde{X}-q+ab)+s\in\Bcal(H_{X,ab},s)$.\\

    {\bf Step 4}: For each pair of distinct hyperplanes chosen from $H_{X,st}, H_{X,ab}, H_{Y,as}, H_{Y,bt}$, we consider the size of their intersection $|H_{\cdot,\cdot}\cap H_{\cdot,\cdot}|$ and sum over the $\binom{4}{2}=6$ terms. It can be seen that
    \begin{enumerate}
        \item each of $a,b,s,t$ appears in exactly 2 hyperplanes and contributes $\binom{2}{2}=1$ to the sum;
        \item each element in $\widetilde{X}\cup\widetilde{Y}$ appears in exactly 3 hyperplanes and contributes $\binom{3}{2}=3$ to the sum: for example, suppose $x\in\widetilde{X}$ thus $x\in H_{X,st}, H_{X,ab}$, by Step 3 we know that $x$ belongs to exactly one of $ H_{Y,as}, H_{Y,bt}$ as well.
    \end{enumerate}
    Therefore the sum is $3(|\widetilde{X}|+|\widetilde{Y}|)+4=3\cdot 2(r-2)+4=6r-8$, but on the other hand, the sum is at most $6(r-2)=6r-12$ by Proposition~\ref{prop:SH_intersection}, a contradiction.
\end{proof}

Now we deal with the second task.

\begin{proposition} \label{prop:XB_XB}
    Let $M$ be a (not necessarily paving) matroid and let $H$ be a stressed hyperplane of $M$ with $\widetilde{M}$ being the relaxation of $M$ at $H$.  
    Let $X,X'$ be subsets of type 0, and $Y,Y'$ be bases of $M$ of type $\geq 2$.  
    Suppose $(X,Y)$ and $(X',Y')$ differ by a single symmetric exchange in $\widetilde{M}$ that is proper ($s\neq t$):
$$
\begin{array}{c c c}
   X   & \swap{s}{t} & X'\\
   Y   & \swap{t}{s} & Y'
\end{array}.
$$

Then there exists an exchange sequence of length $\ell>1$ from $(X,Y)$ to $(X',Y')$ in $\widetilde{M}$ such that all intermediate steps only involve bases in $M$.
\end{proposition}

\begin{proof}
    Consider the fundamental circuit $C(Y,s)$, which cannot be completely contained in the stressed $H$, or otherwise it has size $r$ thus missing only one element from $Y+s$, whereas $Y+s$ contains at least two elements outside of $H$, which is impossible. 
    Hence, we can pick some $a\in C(Y,s)\setminus H$.  
    Similarly, we can pick $b\in C(Y',t)\setminus H$.  
    Note that $a,b$ are not in $X,X'\subset H$.  
    Now the following exchange sequence is what we wanted (if $a=b$, then the middle step does nothing):
    $$
    \begin{array}{c c c c c c c}
        X   & \swap{s}{a} & X-s+a & \swap{a}{b} & X'-t+b & \swap{b}{t} & X'\\
        Y   & \swap{a}{s} & Y-a+s & \swap{b}{a} & Y'-b+t & \swap{t}{b} & Y'
    \end{array}.
    $$
Here $X-s+a\in\Bcal(H,a)$ and $X'-t+b\in\Bcal(H,b)$, while $Y-a+s$ and $Y'-b+t$ are bases of $M$ by the assumption on fundamental circuits.
\end{proof}

\begin{theorem} [Degree 2 case of Theorem~\ref{thm:B}(2)]\label{thm:B2}
    Let $M,\widetilde{M},H$ be as in Proposition~\ref{prop:XB_XB}, and suppose White's conjecture in degree 2 holds for $\widetilde{M}$ while the degree 2 part of {\rm (*)} holds for $M$. Then White's conjecture in degree 2 is true for $M$.
\end{theorem}

\begin{proof}
    Let $(X,Y)$ and $(X',Y')$ be two pairs of bases of $M$ (hence $\widetilde{M}$) with equal multi-set union. We have $|X\setminus H|+|Y\setminus H|=|X'\setminus H|+|Y'\setminus H|\geq 2$.
    
    By the assumption on $\widetilde{M}$, there exists an exchange sequence $(X^{(0)},Y^{(0)})=(X,Y),\ldots,(X^{(\ell)}, Y^{(\ell)})=(X',Y')$ in $\widetilde{M}$, which we assume to have no repeated pairs. Moreover, if $X^{(t)}$ is of type 0, then $Y^{(t)}$ must be of type $\geq 2$, and $Y^{(t-1)},Y^{(t+1)}$ cannot be of type 0; same conclusion with the role of $X,Y$ reversed. Next, by applying Proposition~\ref{prop:XB_XB} to every adjacent pairs $(X^{(t)},Y^{(t)}),(X^{(t+1)},Y^{(t+1)})$ in which both $X^{(t)}, X^{(t+1)}$ (or $Y^{(t)}, Y^{(t+1)}$) are of type 0, we may assume every instance of $(X^{(t)},Y^{(t)})$ in the sequence that involves type 0 subsets is between two pairs $(X^{(t-1)},Y^{(t-1)}), (X^{(t+1)},Y^{(t+1)})$ that only involve bases in $M$, which can be further resolved using (*) and yields an exchange sequence in $M$.
\end{proof}

\begin{corollary}
    White's conjecture in degree 2 is true for paving matroids.
\end{corollary}

\begin{proof}
    By Theorem~\ref{thm:PM_SHR}, any paving matroid is related to $U_{r,E}$ by a sequence of stressed hyperplane relaxations. We induct on the length of such a sequence with the base case being $U_{r,E}$ itself, in which case the statement is trivial. The induction step is the combination of Proposition~\ref{prop:XY_XY} and Theorem~\ref{thm:B2}.
\end{proof}

\section{Degree 3}

We prove the degree 3 part of (*) for paving matroids and another preparation lemma in this section.

\begin{proposition} \label{prop:XYZ_XYZ}
    Let $M$ be a paving matroid and fix an arbitrary hyperplane $H$ of size $\geq r$. Let $\widetilde{M}$ be the relaxation of $M$ at $H$. Let $X,Y,Z,X',Y',Z'$ be 6 bases of $M$ (hence $\widetilde{M}$) such that there exists an exchange sequence from $(X,Y,Z)$ to $(X',Y',Z')$ in $\widetilde{M}$ as
\begin{equation} \label{eq:deg3}
\begin{array}{c c c c c}
   X   & \swap{a}{s} & X''   & \swap{t}{b} & X' \\
   Y & \swap{s}{a} & Y' &             & Y' \\
   Z &             & Z  & \swap{b}{t} & Z'
\end{array}.
\end{equation}

Then there exists an exchange sequence from $(X,Y,Z)$ to $(X',Y',Z')$ in $M$.
\end{proposition}

By Section~\ref{sec:deg2}, we know White's conjecture in degree 2 is true for $M$.

\begin{proof}
    We may assume $X''$ is of type 0 thus $s,t\in H$, or else we are done; in particular, since $X,X'$ are bases of $M$, $a,b\not\in H$. If $s=t,a=b$, then $X=X'$, and it is reduced to the degree 2 case.
    
    We further work out a few ``degenerated'' cases before setting up the double counting proof as in degree 2. 
    
    {\bf Degenerated Case I}: $s=t$ but $a\neq b$, so $X'=X-a+b,Y'=Y-s+a,Z'=Z-b+s$. Apply the symmetric exchange axiom in $M$ with $X,Z$ and $b\in Z\setminus X$ to obtain $p\in X\setminus Z$ such that $X-p+b, Z-b+p\in\Bcal$. If $p=a$, then $(X,Y,Z)$ to $(X',Y,Z-b+a)$ is a valid exchange step in $M$, and from there to $(X',Y',Z')$ is another symmetric exchange. So we assume $p\neq a$, together with $p\neq s$ as $s\not\in X$. Now we have $p\in C(X,b)-ab\subset H$, and the following exchange sequence is valid in $M$:
    $$
    \begin{array}{c c c c c c c}
        X   & \swap{p}{b} & X-p+b & \swap{a}{s} & X'-p+s & \swap{s}{p} & X'\\
        Y   &             & Y    & \swap{s}{a} & Y'      &              & Y'\\
        Z   & \swap{b}{p} & Z-b+p &  =         & Z'-s+p & \swap{p}{s} & Z'
    \end{array}.
    $$
    Here $X'-p+s=(X-ap+s)+b\in\Bcal(H,b)$.

    {\bf Degenerated Case II}: $a=b$ but $s\neq t$, so $X'=X-t+s,Y'=Y-s+a,Z'=Z-a+t$. Apply the symmetric exchange axiom in $M$ with $X,Y$ and $a\in X\setminus Y$ to obtain $q\in Y\setminus X$ such that $X-a+q, Y-q+a\in\Bcal$. Since $X-a\subset H$, $q$ cannot be contained in $H$, in particular it is not $s$ nor $t$, and $q$ remains in $Y'$.

    Next, apply the symmetric exchange axiom with $X',Y'$ and $q\in Y'\setminus X'$ to obtain $p\in X'\setminus Y'$ such that $X'-p+q,Y'-q+p\in\Bcal$. Since $a$ is in $Y'$, $p$ cannot be $a$; since $t\not\in X'$, $p\neq t$. Now the following exchange sequence is valid in $M$:
$$
        \begin{array}{c c c c c c c c c}
           X   & \swap{a}{q} & X - a + q 
                      & \swap{p}{s} & X-ap+qs 
                      & \swap{t}{a} & X'-p+q
                      & \swap{q}{p} & X' \\
           Y & \swap{q}{a} & Y - q + a
                      & \swap{s}{p} & Y-qs+ap
                      &  =          & Y'-q+p
                      & \swap{p}{q} & Y' \\
           Z &             & Z
                      &             & Z
                      & \swap{a}{t} & Z'
                      &             & Z'
        \end{array}.
$$
Here $X-ap+qs=(X-ap+s)+q\in\Bcal(H,q)$; in the case of $p=s$, the second step above simply does nothing.\\

    Hence, we may assume $|\{a,b,s,t\}|=4$ from now on.\\

    {\bf Degenerated Case III}:  $s\not\in Z$ (and by symmetry, the case $t\not\in Y'$). Since the type of $Z'$ is strictly less than that of $Z$ but is still a basis of $M$, $Z$ must be of type $\geq 2$, and since $C(Z,s)$ is missing at most one element from $Z+s$ due to the paving assumption, there exists $q\in C(Z,s)\setminus H$. Now there following exchange sequence is valid in $\widetilde{M}$, with all subsets in the first and last column being bases of $M$:
$$
\begin{array}{c c c c c}
   X   & \swap{a}{s} & X-a+s   & \swap{s}{q} & X-a+q \\
   Y & \swap{s}{a} & Y' &             & Y' \\
   Z &             & Z  & \swap{q}{s} & Z-q+s
\end{array}.
$$
Here $X-a+q\in\Bcal(H,q)$. But this exchange sequence can be viewed as the case $s=t$ considered in Degenerated Case I (or even earlier, if $a=q$). So the middle step can be replaced by an exchange sequence in $M$, while the task of connecting the last column to $(X',Y',Z')$ is reduced to the degree 2 case.\\

{\bf Degenerated Case IV}: $b\in Y$ hence $Y'$ (and by symmetry, the case $a\in Z,Z'$). From Degenerated Case III, we may assume $t\in Y'$ hence $Y$. Consider $C(Y,a)$, which is missing at most one element from $Y+a$, so it either contains $b$ or $t$.

If $b\in C(Y,a)$, then
$$
\begin{array}{c c c c c c c}
   X   & \swap{a}{b} & X-a+b & \swap{b}{s} & X-a+s  & \swap{t}{b} & X' \\
   Y & \swap{b}{a} & Y-b+a & \swap{s}{b} & Y' &             & Y' \\
   Z &             & Z &             & Z  & \swap{b}{t} & Z'
\end{array}.
$$
is a sequence whose first step is a valid exchange in $M$ as $X-a+b\in\Bcal(H,b)$, while the next two steps are valid exchanges in $\widetilde{M}$ with the form considered in Degenerated Case II, so it can be replaced by an exchange sequence in $M$.

If $t\in C(Y,a)$ and $a\in Z$, then
$$
\begin{array}{c c c c c c c}
   X &             & X       & \swap{a}{b} & X-a+b  & \swap{t}{s} & X' \\
   Y & \swap{t}{a} & Y-t+a &            & Y-t+a & \swap{s}{t}  & Y' \\
   Z & \swap{a}{t}  & Z-a+t & \swap{b}{a}  & Z'  &            & Z'
\end{array}.
$$
is a sequence whose last step is a valid exchange in $M$, while the first two steps are either valid exchanges in $M$ (if $Z-a+t$ is a basis of $M$), or are valid changes in the relaxation of $M$ at the hyperplane $\cl(Z-a+t)$ with the form considered in Degenerated Case II (with the row order reversed and elements renamed), so it can be replaced by an exchange sequence in $M$.

If $t\in C(Y,a)$ and $a\not\in Z$, then
$$
\begin{array}{c c c c c c c}
   X & \swap{a}{b} & X-a+b &          & X-a+b  & \swap{t}{s} & X' \\
   Y &            & Y &  \swap{t}{a}  & Y-t+a & \swap{s}{t}  & Y' \\
   Z & \swap{b}{a}  & Z-b+a & \swap{a}{t}  & Z'  &            & Z'
\end{array}.
$$
is a sequence whose last step is a valid exchange in $M$, while the first two steps are either valid exchanges in $M$ already, or at least in its relaxation at $\cl(Z-b+a)$ with the form considered in Degenerated Case I (with the row order permuted and elements renamed), so it can be replaced by an exchange sequence in $M$.\\

    Following the argument in degree 2, we assume $X\cup Y\cup Z=E$ and $X\cap Y\cap Z=\emptyset$.
    For the rest of the proof, we write $X=\widetilde{X}+at, Y=\widetilde{Y}+st, Z=\widetilde{Z}+bs$ (hence $X'=\widetilde{X}+bs, Y'=\widetilde{Y}+at, Z'=\widetilde{Z}+st$) with $b,s\not\in\widetilde{X}, a,b\not\in \widetilde{Y}, a,t\not\in \widetilde{Z}$. We also write $H_{X,st}=H=\cl(\widetilde{X}+st)$ which does not contain $a,b$. Finally, we define $U=\widetilde{X}\cap \widetilde{Y}, V=\widetilde{Y}\cap \widetilde{Z}, W=\widetilde{Z}\cap \widetilde{X}, X^\circ=\widetilde{X}\setminus (\widetilde{Y}\cup \widetilde{Z}), Y^\circ=\widetilde{Y}\setminus (\widetilde{Z}\cup \widetilde{X}), Z^\circ=\widetilde{Z}\setminus (\widetilde{X}\cup \widetilde{Y})$.\\

{\bf Step 1}: We may assume $\widetilde{X}+ab$ is a circuit hence spans a hyperplane $H_{X,ab}$, which necessarily does not contain $s,t$ as $X,X'$ are bases, or otherwise
$$
\begin{array}{c c c c c}
   \widetilde{X}+at & \swap{t}{b} & \widetilde{X}+ab  & \swap{a}{s} & \widetilde{X}+bs \\
   \widetilde{Y}+st &           & \widetilde{Y}+st & \swap{s}{a}  & \widetilde{Y}+at \\
   \widetilde{Z}+bs & \swap{b}{t}  & \widetilde{Z}+st  &            & \widetilde{Z}+st
\end{array}
$$
is a valid exchange sequence. The containment $H_{X,st}\cap H_{X,ab}\supset\widetilde{X}$ is an equality as it is already of size $r-2$.\\

{\bf Step 2}: We may assume $s\not\in C(Z,t)$ (and by symmetry, $t\not\in C(Y',s)$). Otherwise, we may consider the following sequence:
$$
\begin{array}{c c c c c c c}
   \widetilde{X}+at & \swap{t}{s} & \widetilde{X}+as &          & \widetilde{X}+as  & \swap{a}{b} & \widetilde{X}+bs \\
   \widetilde{Y}+st &            & \widetilde{Y}+st &  \swap{s}{b}  & \widetilde{Y}+bt & \swap{b}{a}  & \widetilde{Y}+at \\
   \widetilde{Z}+bs & \swap{s}{t}  & \widetilde{Z}+bt & \swap{b}{s}  & \widetilde{Z}+st  &            & \widetilde{Z}+st
\end{array},
$$
here $\widetilde{X}+as\in\Bcal(H_{X,st},a)$. This sequence is either already valid in $M$ (if $\widetilde{Y}+bt$ is a basis of $M$), or the last two steps are at least valid in the relaxation of $M$ at the hyperplane $\cl(\widetilde{Y}+bt)$, with the form considered in Degenerated Case I.\\

{\bf Step 3}: From Step 2, we may assume $C(Z,t)=Z+t-s=\widetilde{Z}+bt, C(Y',s)=Y'+s-t=\widetilde{Y}+as$ are circuits that span two hyperplanes $H_{Z,bt}, H_{Y,as}$, and $s\not\in H_{Z,bt}$ as $Z,Z'$ are bases (and by symmetry, $t\not\in H_{Y,as}$). We may further assume $a\in H_{Z,bt}$ (and by symmetry, $b\in H_{Y,as}$), for otherwise
$$
\begin{array}{c c c c c}
   \widetilde{X}+at & \sim & \widetilde{X}+bs  &          & \widetilde{X}+bs \\
   \widetilde{Y}+st &            & \widetilde{Y}+st & \swap{s}{a}  & \widetilde{Y}+at \\
   \widetilde{Z}+bs & \sim  & \widetilde{Z}+at  & \swap{a}{s}   & \widetilde{Z}+st
\end{array}
$$
is a valid exchange sequence, here $\widetilde{Z}+at\in\Bcal(H_{Z,bt},a)$ and we make use of the validity of White's conjecture in degree 2 for the first step. Therefore we shall write $H_{Z,abt}, H_{Y,abs}$ instead. The size $r-1$ independent sets $\widetilde{Y}+t, \widetilde{Z}+s$ also span two hyperplanes $H_{Y,t}, H_{Z,s}$ respectively. Using a similar cardinality argument, we have $H_{Y,abs}\cap H_{Y,t}=\widetilde{Y}, H_{Z,abt}\cap H_{Z,s}=\widetilde{Z}$.\\

{\bf Step 4}: We may assume $E=H_{X,st}\cup H_{X,ab}$, or otherwise we may pick $p\in E\setminus(H_{X,st}\cup H_{X,ab})$, say $p\in Y^\circ\cup V$ (the case $p\in Z^\circ\cup V$ is similar), and the following exchange sequence is valid:
$$
\begin{array}{c c c c c c c}
   \widetilde{X}+at & \swap{a}{p} & \widetilde{X}+tp & \swap{t}{b} & \widetilde{X}+bp  & \swap{p}{s} & \widetilde{X}+bs \\
   \widetilde{Y}+st &  \swap{p}{a}  & \widetilde{Y}-p+ast &         & \widetilde{Y}-p+ast & \swap{s}{p}  & \widetilde{Y}+at \\
   \widetilde{Z}+bs &          & \widetilde{Z}+bs &  \swap{b}{t}  & \widetilde{Z}+st  &         & \widetilde{Z}+st
\end{array},
$$
here $\widetilde{X}+tp\in\Bcal(H_{X,st},p), \widetilde{X}+bp\in\Bcal(H_{X,ab},p), (\widetilde{Y}-p+as)+t\in\Bcal(H_{Y,abs},t)$.\\

{\bf Step 5}: We may assume $E=H_{Y,abs}\cup H_{Y,t}$ (and by symmetry, $E=H_{Z,abt}\cup H_{Z,s}$), or otherwise we may pick $p\in E\setminus(H_{Y,abs}\cup H_{Y,t})$. Such a $p$ is either in $Z^\circ\cup W$, in which case the following exchange sequence is valid: 
$$
\begin{array}{c c c c c c c c c}
   \widetilde{X}+at &           & \widetilde{X}+at & \swap{t}{s} & \widetilde{X}+as  &  \swap{a}{b} & \widetilde{X}+bs &           & \widetilde{X}+bs \\
   \widetilde{Y}+st & \swap{t}{p}  & \widetilde{Y}+sp & \swap{s}{t} & \widetilde{Y}+tp &         & \widetilde{Y}+tp & \swap{p}{a} & \widetilde{Y}+at\\
   \widetilde{Z}+bs & \swap{p}{t}  & \widetilde{Z}-p+bst &         & \widetilde{Z}-p+bst  & \swap{b}{a} & \widetilde{Z}-p+ast & \swap{a}{p} & \widetilde{Z}+st
\end{array},
$$
here $\widetilde{X}+as,\widetilde{X}+bs\in\Bcal(H_{X,ab},s), \widetilde{Y}+sp\in\Bcal(H_{Y,abs},p), \widetilde{Y}+tp\in\Bcal(H_{Y,t},p), (\widetilde{Z}-p+bt)+s,(\widetilde{Z}-p+at)+s\in\Bcal(H_{Z,abt},s)$.

Or else, $p$ is in $X^\circ$, and the following exchange sequence is valid:
$$
\begin{array}{c c c c c c c}
   \widetilde{X}+at & \swap{p}{s} & \widetilde{X}-p+ast & \swap{t}{b} & \widetilde{X}-p+abs  & \swap{a}{p} & \widetilde{X}+bs \\
   \widetilde{Y}+st &  \swap{s}{p}  & \widetilde{Y}+tp &         & \widetilde{Y}+tp & \swap{p}{a}  & \widetilde{Y}+at \\
   \widetilde{Z}+bs &          & \widetilde{Z}+bs &  \swap{b}{t}  & \widetilde{Z}+st  &         & \widetilde{Z}+st
\end{array},
$$
here $(\widetilde{X}-p+st)+a\in\Bcal(H_{X,st},a), (\widetilde{X}-p+ab)+s\in\Bcal(H_{X,ab},s), \widetilde{Y}+tp\in\Bcal(H_{Y,t},p)$.\\

{\bf Step 6}: For each pair of hyperplanes chosen from $H_{X,st}, H_{X,ab}, H_{Y,abs}, H_{Y,t}, H_{Z,abt}, H_{Z,s}$, we consider the size of their intersection $|H_{\cdot,\cdot}\cap H_{\cdot,\cdot}|$ and sum over the $\binom{6}{2}=15$ terms. It can be seen that
    \begin{enumerate}
        \item each of $a,b,s,t$ appears in exactly 3 hyperplanes and contributes $\binom{3}{2}=3$ to the sum;
        \item each element in $X^\circ\cup Y^\circ\cup Z^\circ$ appears in exactly 4 hyperplanes and contributes $\binom{4}{2}=6$ to the sum: for example, suppose $y\in Y^\circ$ thus $y\in H_{Y,abs}, H_{Y,t}$, by Step 4 and 5 we know that $y$ belongs to exactly one of $H_{X,st}, H_{X,ab}$ and exactly one of $H_{Z,abt}, H_{Z,s}$ as well.
        \item each element in $U\cup V\cup W$ appears in exactly 5 hyperplanes and contributes $\binom{5}{2}=10$ to the sum: for example, suppose $u\in U\subset X,Y$ thus $u\in H_{X,st}, H_{X,ab}, H_{Y,abs}, H_{Y,t}$, by Step 5 we know that $u$ belongs to exactly one of $H_{Z,abt}, H_{Z,s}$ as well.
    \end{enumerate}
    Therefore the sum is $6(|X^\circ|+Y^\circ|+|Z^\circ|)+10(|U|+|V|+|W|)+12\geq 5[(|X^\circ|+|U|+|W|)+(|Y^\circ|+|U|+|V|)+(|Z^\circ|+|V|+|W|)]+12=5(|\widetilde{X}|+|\widetilde{Y}|+|\widetilde{Z}|)+12=5\cdot 3(r-2)+12=15r-18$, but on the other hand, the sum is at most $15(r-2)=15r-30$ by Proposition~\ref{prop:SH_intersection}, a contradiction.
\end{proof}

\begin{proposition} \label{prop:XBB_XBB}
    Let $M,\widetilde{M},H$ be as in Theorem~\ref{thm:B2}. Let $Y, Z,Y', Z'$ be bases of $M$ such that $(Y,Z)$ and $(Y', Z')$ differ only by a symmetric exchange (with $y\neq z$)
    $$
    \begin{array}{c c c}
       X   &             & X\\
       Y   & \swap{y}{z} & Y'\\
       Z   & \swap{z}{y} & Z'
    \end{array},
    $$
and let $X$ be of type 0. Suppose the total type of $Y,Z$ (hence $Y',Z'$) is at least $3$.
    
Then there exists an exchange sequence of length $\ell>1$ from $(X,Y,Z)$ to $(X,Y',Z')$ in $\widetilde{M}$ such that all intermediate steps only involve bases in $M$.
\end{proposition}

\begin{proof}
If the total type is 3, then exactly one of $Y,Z$ (respectively, $Y', Z'$) is of type 2 whereas the other 1. We can, on each side, apply Lemma~\ref{lem:type02} to exchange an element not in $H$ in the type 2 basis with an element in $X$ to make all three bases to be of type 1, in which case the statement follows from Corollary~\ref{cor:type1-WC-bases}.

Hence, we may assume the total type is at least 4 and $\type (Y)\geq 2$ for the rest of the proof.

{\bf Case I}: $y,z\not\in H$. By Lemma~\ref{lem:type02}, we can pick $p\in X \setminus Y$ such that $X - p + y , Y - y + p$ are bases of $M$. We also have $X-p+z$ is of type 1 and the following exchange sequence is valid:
$$
        \begin{array}{c c c c c c c}
           X   & \swap{p}{y} & X - p + y & \swap{y}{z} & X - p + z & \swap{z}{p} & X \\
           Y   & \swap{y}{p} & Y - y + p &   =        & Y' - z + p & \swap{p}{z} & Y' \\
           Z   &             & Z         & \swap{z}{y} & Z'         &             & Z'
        \end{array}.
$$

{\bf Case II}: $y \not\in H, z \in H$ (and by symmetry, the case $y \in H, z \not\in H$). As in Case I, we first pick $p\in X$ such that $X - p + y , Y - y + p \in \Bcal$. If $z\not\in X$, then $z\in Z \setminus (X - p + y)$ and we can pick $q\in (X - p + y) \setminus Z$ such that $X-pq+yz, Z - z + q \in \Bcal$; note that $q\neq y$ as $X-p+z$ is of type 0. The following exchange sequence is thus valid.

$$
        \begin{array}{c c c c c c c c c}
           X   & \swap{p}{y} & X - p + y & \swap{q}{z} & X - pq + yz & \swap{z}{p} & X - q + y & \swap{y}{q} & X\\
           Y   & \swap{y}{p} & Y - y + p &      =      & Y' - z + p & \swap{p}{z} & Y'         &             & Y'\\
           Z   &             & Z         & \swap{z}{q} & Z - z + q & =            & Z' - y + q & \swap{q}{y} & Z'
        \end{array},
$$
here $X-q+y\in\Bcal(y)$. If $z\in X$, then we take $q=z$ in the above sequence and omit the second step (as well as the third step if $z=p$).\\

{\bf Case III}: $y, z \in H$. Again applying Lemma~\ref{lem:type02}, we pick some $c\in Y\setminus H$, which cannot be $y,z$, and $p\in X\setminus Y$ such that $X-p+c,Y-c+p\in\Bcal$.

If $y,z\not\in X$, then $y\in (Y-c+p)\setminus (X-p+c)$, and there exists $q\in (X-p+c)\setminus (Y-c+p)$ such that $X-pq+cy, Y-cy+pq\in\Bcal$; note that $q\neq c$ as $X-p+y$ is of type 0, so $q\in X\subset H$. Now the following exchange sequence is valid:
$$
        \begin{array}{c c c c c c c c c}
           X   & \swap{p}{c} & X - p + c 
               & \swap{q}{y} & X - pq + cy 
               & \swap{y}{z} & X - pq + cz & \sim & X\\ 
           Y   & \swap{c}{p} & Y - c + p 
               & \swap{y}{q} & Y - cy + pq 
               &             & Y - cy + pq & \sim & Y'\\
           Z   &             & Z 
               &             & Z 
               & \swap{z}{y} & Z'  &     & Z'
        \end{array},
$$
here $X - pq + cz\in\Bcal(c)$ and the last step makes use of the validity of White's conjecture in degree 2 guaranteed by Theorem~\ref{thm:B2}: since $\type(Y')=\type(Y)\geq 2$, $(X,Y')$ is a symmetric exchange away from a pair of bases of $M$ by Lemma~\ref{lem:type02}, which can be further related to the fourth column above by symmetric exchanges.

If $y\in X, z\not\in X$ (and by symmetry, the case $z\not\in X, y\in X$), then $y$ is already in $X-p+c$. Hence we may set $q=y$ in the second step and let it does nothing, the sequence remains valid.

For the last case $y,z\in X$, note that $C(Y,z)$ cannot be completely in $H$, otherwise it is of size $r$ and misses only one element from $Y+z$, while $Y+z$ contains $\type(Y)\geq 2$ elements not in $H$. So there exists $c\in C(Y,z)\setminus H$, and $Y - c + z \in \Bcal$ by the choice of $c$. Thus the following exchange sequence is valid:
$$
    \begin{array}{c c c c c c c}
        X   & \swap{z}{c}  & X - z + c 
            & \swap{y}{z}  & X - y + c 
            & \swap{c}{y}  & X \\
        Y   & \swap{c}{z}  & Y - c + z 
            & =        & Y' - c + y 
            & \swap{y}{c}  & Y' \\
        Z   &              & Z 
            & \swap{z}{y}  & Z' 
            &              & Z'
        \end{array},
$$
here $X-z+c,X-y+c\in\Bcal(c)$.
\end{proof}

\section{Higher Degree}

We prove Theorem~\ref{thm:B} in this section, which also wraps up the proof of Theorem~\ref{thm:A}. We assume $H$ is a (size $\geq r$) stressed hyperplane of a (not necessarily paving) matroid $M$ and let $\widetilde{M}$ be its relaxation at $H$.

\subsection{White's Conjecture for $M$ Implies White's Conjecture for $\widetilde{M}$}

\begin{proof} [Proof of Theorem~\ref{thm:B}(1)]
Suppose $(B_1,\dots,B_n)$ and $(B'_1,\dots,B'_n)$ are two tuples of bases of $\widetilde{M}$ such that they have the same multi-set union.  
We aim to show that there exists a sequence of symmetric exchanges transforming $(B_1,\dots,B_n)$ into $(B'_1,\dots,B'_n)$.

By repeatedly applying Lemma~\ref{lem:type02}, we may assume each tuple either does not contain any type 0 bases or does not contain any type $\geq 2$ bases. Since the total type is preserved by symmetric exchanges, the two tuples belong to the same case.\\

{\bf Case I}: There are no type 0 bases. In this case, all $B_i$'s and $B'_i$'s are bases of $M$, and there exists an exchange sequence using bases of $M$ (hence $\widetilde{M}$) by assumption.\\

{\bf Case II}: There are no type $\geq 2$ bases. This is Lemma~\ref{lem:type1-WC}.
\end{proof}

\subsection{Setup for Induction in the Opposite Direction} \label{sec:induct}

Suppose White's conjecture is true for $\widetilde{M}$, and (*) holds for $M$. By Section~\ref{sec:deg2}, we know White's conjecture in degree 2 is true for $M$. For $n\geq 3$, fix a pair of tuples $(B_1,\ldots,B_n), (B'_1,\ldots,B'_n)$ of bases of $M$ with equal multi-set union. As an easy observation, the total type of $B_j$'s (respectively, $B'_j$'s) is at least $n$ as each basis is of type $\geq 1$.

By hypothesis, there exists a valid exchange sequence in $\widetilde{M}$ that connects the tuples: $$\mathfrak{B}=((B_1^{(0)},\ldots,B_n^{(0)}), (B_1^{(1)},\ldots,B_n^{(1)}),\ldots,(B_1^{(\ell)},\ldots,B_n^{(\ell)})),$$
in which we assume there are no repeated tuples.
Define $\Hcal_i(\mathfrak{B}):=\{1\leq j\leq\ell: B_i^{(j)}\subset H\}$ for each $i=1,\ldots,n$ and $\Rcal(\mathfrak{B}):=\{1\leq i\leq n:\Hcal_i(\mathfrak{B})\neq\emptyset\}$. By definition, when $\Rcal(\mathfrak{B})=\emptyset$, $\mathfrak{B}$ itself is a valid exchange sequence only using bases in $M$. So by induction on $|\Rcal(\mathfrak{B})|$, if $\Rcal(\mathfrak{B})\neq\emptyset$, it suffices to fix an arbitrary $i\in\Rcal(\mathfrak{B})$ and construct a new exchange sequence $\mathfrak{B}'$ such that $\Rcal(\mathfrak{B}')\subset\Rcal(\mathfrak{B})\setminus\{i\}$. By permuting indices if necessary, we assume $i=1$ for the rest of the proof.

\subsection{Reducing $\Hcal_1$ into Disjoint Singletons}

We want to modify $\mathfrak{B}$, without changing $\Rcal$, so that $\Hcal_1$ contains no consecutive indices. Suppose $\{l,l+1\}\subset\Hcal_1(\mathfrak{B})$. We aim to replace the $(l+1)$-th exchange step by a longer exchange sequence in $\widetilde{M}$ with the same starting and ending tuples (and re-index the steps if necessary), such that
\begin{enumerate}
    \item every $B_1^{(\bullet)}$ of the internal steps is either in $M$, or is at least 2 steps apart from another instance of $B_1^{(\bullet)}$ of type 0 (in all cases except Case II(d), the former is guaranteed), so in the updated $\Hcal_1$, between the two indices originally corresponding to $l,l+1$ (inclusively), any two indices are now at least 2 steps apart;
    \item for every other position $k$, the bases in the internal steps are either bases of $M$, or $k\in\Rcal$ to begin with (in all cases except Case II(c), the latter follows trivially from construction that these $B_k^{(\bullet)}$'s are equal to either $B_k^{(l)}$ or $B_k^{(l+1)}$), so the replacement does not change $\Rcal$.
\end{enumerate}

{\bf Case I}: The symmetric exchange between the $l$-th and the $(l+1)$-th tuple involves the first position, say

$$
\begin{array}{c c c}
   X:=B_1^{(l)} & \swap{s}{t} & X':=B_1^{(l+1)} \\
   Y:=B_k^{(l)} & \swap{t}{s} & Y':=B_k^{(l+1)}
\end{array}.
$$

$X,X'$ are by definition of type 0 and $s,t\in H$, so $\type(Y)=\type(Y')$.

{\bf Case I(a)}: $\type(Y)\geq 2$. In which case we replace this exchange step by the sequence described in Proposition~\ref{prop:XB_XB}.

{\bf Case I(b)}: $\type(Y)<2$. By the pigeonhole principle, there exists $k'$ such that the type of $Z:=B_{k'}^{(l)}=B_{k'}^{(l+1)}$ is at least 2. Pick $q\subset Z\setminus H$ and $p\in X\setminus Z$ such that $X-p+q,Z-q+p\in\Bcal$ using Lemma~\ref{lem:type02}. Note that $q\neq s,t$ as $s,t\in H$ but $q\not\in H$, and $p\neq t$ as $p\in X, t\not\in X$. If $p\neq s$, then the following replacement exchange sequence suffices:
$$
\begin{array}{c c c c c c c}
   X & \swap{p}{q} & X-p+q & \swap{s}{t} & X'-p+q & \swap{q}{p}   & X' \\
   Y &            & Y &  \swap{t}{s}  & Y' &            & Y' \\
   Z & \swap{q}{p}  & Z-q+p &          & Z-q+p  & \swap{p}{q}  & Z
\end{array},
$$
here $X'-p+q\in\Bcal(q)$ is a basis of $M$. A similar sequence from the other direction can be constructed if upon considering the same $q\in Z\setminus H$, the corresponding $p'\in X'\setminus Z$ such that $X'-p'+q,Z-q+p'\in\Bcal$ is not $t$. Hence, we may assume $X-s+q,Z-q+s,X'-t+q,Z-q+t$ are all bases of $M$, and the following replacement exchange sequence suffices:
$$
\begin{array}{c c c c c c c}
   X & \swap{s}{q} & X-s+q & =      & X'-t+q & \swap{q}{t}   & X' \\
   Y &            & Y &  \swap{t}{s}  & Y' &            & Y' \\
   Z & \swap{q}{s}  & Z-q+s & \swap{s}{t}  & Z-q+t  & \swap{t}{q}  & Z
\end{array}.
$$

{\bf Case II}: The symmetric exchange between the $l$-th and the $(l+1)$-th tuple does not involve the first position, say
$$
\begin{array}{c c c}
   X:=B_1^{(l)} &            & X=B_1^{(l+1)} \\ 
   Y:=B_k^{(l)} & \swap{s}{t} & Y':=B_k^{(l+1)} \\
   Z:=B_{k'}^{(l)}           & \swap{t}{s} & Z':=B_{k'}^{(l+1)}
\end{array}.
$$

{\bf Case II(a)}: $\type(Y)+\type(Z)<3$. By the pigeonhole principle, there exists $k''\neq k,k'$ such that the type of $W:=B_{k''}^{(l)}=B_{k''}^{(l+1)}$ is at least $2$. Apply Lemma~\ref{lem:type02} and pick $q\in W\setminus H$ and $p\in X\setminus W$ such that $X':=X-p+q,W':=W-q+p\in\Bcal$. Then we replace the exchange step by $(X,Y,Z,W)\rightarrow (X',Y,Z,W')\rightarrow(X',Y',Z',W')\rightarrow(X,Y',Z',W)$.

{\bf Case II(b)}: $\type(Y)+\type(Z)\geq 3$, and by assuming $\type(Y)\leq\type(Z)$ without loss of generality, $\type(Y),\type(Y')\geq 1$. In which case $\type(Z)\geq 2$ thus $\type(Z')\geq 1$, so $Y,Y',Z,Z'\in\Bcal$, and we can replace this exchange step by the sequence described in Proposition~\ref{prop:XBB_XBB}.

{\bf Case II(c)}: $\type(Y)+\type(Z)\geq 3$, and by assuming $\type(Y)\leq\type(Z)$ as in Case II(b), $\type(Y)=0,\type(Y')=1$ thus $\type(Z)\geq 3$ (and by symmetry, the case $\type(Y')=0,\type(Y)=1$). We have $s\in H, t\not\in H$. Since $t\in Z\setminus H$, we pick $p\in X\setminus Z$ such that $X-p+t, Z-t+p\in\Bcal$ by Lemma~\ref{lem:type02}. If $p\not\in Y'$, then the following exchange sequence works:

$$
\begin{array}{c c c c c c c}
   X & \swap{p}{t} & X-p+t &       & X-p+t & \swap{t}{p}   & X \\
   Y &            & Y &  \swap{s}{p}  & Y-s+p & \swap{p}{t}   & Y' \\
   Z & \swap{t}{p}  & Z-t+p & \swap{p}{s}  & Z' &          & Z'
\end{array}.
$$
Notice that $Y-s+p$ in the third column is of type 0, but this does not expand $\Rcal$ as $Y$ is already of type 0. Also, if $p=s$, then the second step simply does nothing.

If $p\in Y'$ (hence $Y$), then $p\neq s$ as $s\not\in Y'$. Since $t\in Y'\setminus H\subset Y'\setminus X$, we can pick $q\in X\setminus Y'$ such that $X-q+t, Y'-t+q$ are bases of $\widetilde{M}$, in which case $X-q+t$ is of type 1 thus a basis of $M$, whereas $Y'-t+q$ is of type 0 but does not expand $\Rcal$ as explained above. If $q=s$, then the following exchange sequence works:
$$
    \begin{array}{c c c c c c c}
        X & \swap{p}{t} & X - p + t 
                      & \swap{s}{p} & X - s + t 
                      & \swap{t}{s} & X\\
        Y &             & Y
                       &             & Y
                      & \swap{s}{t} & Y' \\
        Z & \swap{t}{p} & Z-t+p 
                      & \swap{p}{s} & Z' 
                      &             & Z'
    \end{array}.
$$
If $q\neq s$, then $q\in X\setminus Y'=X\setminus (Y-s+t)$ implies $q\not\in Y$ and $q\neq p$, and the following exchange sequence works:
$$
    \begin{array}{c c c c c c c c c}
        X & \swap{p}{t} & X-p+t & \swap{q}{p} & X-q+t &             & X-q+t & \swap{t}{q} & X\\
        Y &             & Y    & \swap{p}{q} & Y-p+q & \swap{s}{p} & Y'-t+q & \swap{q}{t} & Y'\\
        Z  & \swap{t}{p} & Z-t+p &           & Z-t+p & \swap{p}{s} & Z'     &             & Z'
        \end{array},
$$
here both $p,q$ are chosen from $X\subset H$, so $Y-p+q$ is of type 0 (but does not expand $\Rcal$).

{\bf Case II(d)}: $\type(Y)+\type(Z)\geq 3$, and by assuming $\type(Y)\leq\type(Z)$ as in Case II(b), $\type(Y)=\type(Y')=0$. Since $\type(Z)=\type(Z')\geq 3$, we first apply Proposition~\ref{prop:XB_XB} with $(Y,Z),(Y',Z')$ to obtain an exchange sequence $(Y^{(0)},Z^{(0)})=(Y,Z),(Y^{(1)}, Z^{(1)}),\ldots,(Y^{(\ell')},Z^{(\ell')})=(Y',Z')$ in $\widetilde{M}$ such that $\ell'>1$ and all intermediate steps only involves bases in $M$. Now each step in the exchange sequence $(X,Y^{(0)},Z^{(0)})=(X,Y,Z),(X,Y^{(1)}, Z^{(1)}),\ldots,(X,Y^{(\ell')},Z^{(\ell')})=(X,Y',Z')$ can be further replaced by a longer exchange sequence whose first position are all bases of $M$ in the internal steps, using Case II(b) or II(c).

\subsection{Eliminating $\Hcal_1$}

In the final step, every occurrence of type 0 subset $B_1^{(j)}$ in the first position of $\mathfrak{B}$ is surrounded by bases $B_1^{(j-1)}, B_1^{(j+1)}$ of $M$ (necessarily of type 1). In particular, both the $j$-th and the $(j+1)$-th exchange steps involve the first position. Our task is to replace the two said steps by an exchange sequence valid in $\widetilde{M}$, such that all subsets in the first position in the sequence are bases of $M$, while not expanding $\Rcal$. By the time we finished processing all $j\in\Hcal_1$, 1 is no longer contained in $\Rcal$, as desired.\\

{\bf Case I}: Both steps involve the same two positions, say the first and the $k$-th. If $B_k^{(j-1)}, B_k^{(j+1)}$ are bases of $M$, then we may apply the degree 2 part of (*). Otherwise, say $B_k^{(j-1)}$ is of type 0, then there exists a unique $a\in (B_1^{(j-1)}\cup B_k^{(j-1)})\setminus H$ and it is in $B_1^{(j-1)}$, $B_k^{(j)}$ and $B_1^{(j+1)}$. That is, the exchange sequence looks like this:
$$
\begin{array}{c c c c c}
   B_1^{(j-1)}   & \swap{a}{s} & B_1^{(j)} & \swap{t}{a} & B_1^{(j+1)}\\
   B_k^{(j-1)}   & \swap{s}{a} & B_k^{(j)} & \swap{a}{t} & B_k^{(j+1)}
\end{array}.
$$
We can then merge the exchanges into one:
$$
\begin{array}{c c c}
   B_1^{(j-1)}   & \swap{t}{s} & B_1^{(j+1)}\\
   B_k^{(j-1)}   & \swap{s}{t} & B_k^{(j+1)}
\end{array}.
$$

{\bf Case II}: The $j$-th step is between the first and the $k$-th position while the $(j+1)$-th is between the first and the $k'$-th, i.e.,
$$
    \begin{array}{c c c c c}
       X:=B_1^{(j-1)}   & \swap{a}{s} & X'':=B_1^{(j)}   & \swap{t}{b} & X':=B_1^{(j+1)} \\
       Y:=B_k^{(j-1)}   & \swap{s}{a} & Y':=B_k^{(j)} &             & Y' \\
       Z:=B_{k'}^{(j-1)}   &             & Z  & \swap{b}{t} & Z':=B_{k'}^{(j+1)}
    \end{array},
$$
with $s,t,\in H, a,b\not\in H$. Moreover, $b\not\in X$ unless $a=b$, and $t\in X$ unless $s=t$.

If $Y,Y',Z,Z'$ are all bases of $M$, then we may apply the degree 3 part of (*). So we assume some of them are of type 0, which cannot be $Y'$ nor $Z$, as they contain $a$ and $b$, respectively.\\

{\bf Case II(a)}: $Y$ is of type 0 but not $Z'$ (and the case $\type(Y)\neq 0=\type(Z')$ by symmetry). We have $b\not\in Y$. If $s=t$, then the following exchange sequence works:
$$
    \begin{array}{c c c c c c c}
        X   &             & X 
            & \swap{a}{b} & X' \\
        Y & \swap{s}{b} & Y-s+b = Y' - a + b
            & \swap{b}{a} & Y' \\
        Z   & \swap{b}{s} & Z'  
            &             & Z'
    \end{array},
$$
here $Y-s+b\in\Bcal(b)$, and if $a=b$, then the second step does nothing. So we assume $s\neq t$ for the rest of Case II(a), and we know $X-t+s=X'-b+a\in\Bcal(a), Y'-a+b\in\Bcal(b)$ are bases of $M$.

Suppose $t\not\in Y$. Then the following exchange sequence works:
$$
    \begin{array}{c c c c c c c}
        X   & \swap{t}{s} & X' - b + a 
            &             & X' - b + a 
            & \swap{a}{b} & X' \\
        Y & \swap{s}{t} & Y - s + t 
            & \swap{t}{b} & Y' - a + b 
            & \swap{b}{a} & Y' \\
        Z   &             & Z 
            & \swap{b}{t} &  Z'  
            &             & Z'
    \end{array}.
$$
Here $Y-s+t$ is of type 0, but so does $Y$, and $\Rcal$ is not expanded in the process. In the case of $a=b$, the last step does nothing.

Suppose $t\in Y$ instead. Then $t\in (Y'-a+b)\setminus (X'-b+a)$, so we may pick $q\in (X'-b+a)\setminus (Y'-a+b)$ such that $X'-bq+at=X-q+s, Y'-at+bq$ are bases of $M$; in particular, $q\neq a$ as $X'-b+t$ is of type 0, so $q\in X'-b\subset X+s-a$, and $q\not\in Y$ unless $q=s$. Now the following exchange sequence works:
$$
    \begin{array}{c c c c c c c c c}
        X   & \swap{q}{s} & X - q + s 
            &                & X - q + s & \swap{t}{q} & X' - b + a 
            & \swap{a}{b}    & X'\\
        Y & \swap{s}{q} & Y - s + q 
            & \swap{t}{b}    & Y' - at + bq & \swap{q}{t} & Y' - a + b 
            & \swap{b}{a}    & Y'\\
        Z   &                & Z
            & \swap{b}{t}    & Z' &                & Z'
                   &                & Z'
    \end{array},
$$
here $Y-s+q$ is of type 0 (but does not change $\Rcal$ as explained above). In the case of $q=s$, the first step does nothing so $X-q+s=X$ in the above does not contain $s$ after all; in the not mutually exclusive case $a=b$, the last step does nothing.\\

{\bf Case II(b)}: Both $Y,Z'$ are of type 0. We have $a\not\in Z$ unless $a=b$. If $s\not\in Z'$, then $s\neq t\in Z'$ thus $s\not\in Z$, and the following exchange sequence works:
$$
    \begin{array}{c c c c c c c}
        X   & \swap{a}{b} & X - a + b   &              & X - a + b  &      \swap{t}{s} & X'  \\
        Y  &             & Y            & \swap{s}{a}  & Y'               &             & Y'   \\
        Z   & \swap{b}{a} & Z - b + a    & \swap{a}{s}  & Z - b + s        & \swap{s}{t} & Z'  
    \end{array},
$$
here $X-a+b\in\Bcal(b), Z-b+a=Z'-t+a\in\Bcal(a)$, and $Z-b+s$ is of type 0 (but does not expand $\Rcal$ as $Z'$ already is). In the case of $a=b$, the first step does nothing.

If $s\in Z'$, then $s\in Z$ unless $s=t$, and the following exchange sequence works:
$$
    \begin{array}{c c c c c c c}
        X   & \swap{a}{b} & X - a + b & \swap{t}{s}  & X' &             & X' \\
        Y  &             & Y          &              & Y          & \swap{s}{a}  & Y'  \\
        Z   & \swap{b}{a} & Z - b + a & \swap{s}{t}  & Z' - s + a & \swap{a}{s}  & Z' 
    \end{array},
$$
here $X-a+b\in\Bcal(b), Z-b+a=Z'-t+a, Z'-s+a\in\Bcal(a)$. In the case of $a=b$ (and/or $s=t$), the first (and/or second) step does nothing.\\

As explained in Section~\ref{sec:induct}, Theorem~\ref{thm:B}(2) (for degree $\geq 3$) is now proven.

\begin{proof} [Proof of Theorem~\ref{thm:A}]
    We apply the same inductive proof as in the degree 2 case in Section~\ref{sec:deg2}, using Proposition~\ref{prop:XYZ_XYZ} and the general version of Theorem~\ref{thm:B}(2).
\end{proof}

\section{Open Problems}

We list a few research directions that could be follow-up of our work.

As mentioned in the introduction, paving matroids are special cases of the class of split matroids, and these matroids are closely related to the notion of {\em stressed subset} \cite{Ferroni2024}. A subset $A$ (of rank $r'$) of a matroid $M$ is {\em stressed} if $M|_A, M/A$ are both uniform, and we can {\em relax} $M$ at $A$ by declaring every subset $S$ of size $r$ satisfying $|S\cap A|\geq r'+1$ to be a basis (together with the original bases of $M$); a stressed hyperplane is thus a stressed subset of rank $r-1$. A(n elementary) split matroid can be characterized as a matroid that becomes uniform after relaxing all of its (non-trivial) stressed subsets \cite[Theorem~4.8]{Ferroni2024}. It may be possible to prove White's conjecture for split matroids following an inductive argument on the length of such a relaxation sequence.

More generally, as suggested by Mateusz Micha\l ek, one may look for other matroidal operations (besides the trivial ones such as direct sum) that preserve the validity of White's conjecture.

With our inductive component (Theorem~\ref{thm:B}) works beyond paving matroids, it is also interesting to look for other families of matroids related by stressed hyperplane relaxations, where (*) can be verified for all of their members while White's conjecture can be verified for any of their members.

Another direction is to strengthen Theorem~\ref{thm:A} to other versions of White's conjecture. On the combinatorial side, we ask what is the minimum $\ell$ that guarantees an exchange sequence exist? On the algebraic/geometric side, whether any of our techniques can be applied to prove some of the stronger variants, for example, whether one can produce suitable monomial order (perhaps using the induction process) so that the symmetric exchange binomials form a Gr\"{o}bner basis of $I_M$?

\bibliographystyle{alpha}  % Choose a bibliography style
\bibliography{ref}  % Specify your .bib file name 

\end{document}